\documentclass[11pt, reqno]{amsart}
\usepackage{fullpage}
\usepackage{tvhpreamble}

\usepackage[scr=boondoxo]{mathalpha}
\usepackage{graphicx} 
\newcommand{\WF}{\mathrm{WF}}
\newcommand{\scS}{\prescript{\rm{sc}}{}{S}}
\newcommand{\scWF}{\prescript{\rm{sc}}{}{\WF}}
\newcommand{\cWF}{\prescript{\rm{c}}{}{\WF}}
\newcommand{\FWF}{\prescript{\F}{}{\WF}}
\newcommand{\oscWF}{\prescript{\rm{osc}}{}{\WF}}
\newcommand{\oscWFt}{\prescript{\rm{osc}}{}{\widetilde{\WF}}}
\newcommand{\cNDelta}{\prescript{\rm{c}}{}{N_\Delta}}
\newcommand{\oscNDelta}{\prescript{\rm{osc}}{}{N_\Delta}}
\newcommand{\FWFt}{\prescript{\F}{}{\widetilde{\WF}}}
\newcommand{\resizemid}{\;\middle\vert\;}
\newcommand{\tauWF}{\prescript{\tau}{}{\WF}}
\DeclareMathOperator{\Op}{\mathrm{Op}}

\newcommand{\D}{\mathscr{D}}
\newcommand{\Ddir}{\mathscr{D}^{\rm{dir}}}
\newcommand{\sP}{\mathsf{P}}
\DeclareMathOperator{\Char}{\mathrm{Char}}
\title[Global Wavefront Sets and Paraproducts]{A Global Wavefront Set Condition for Defining Bony's Paraproduct Decomposition}

\author[J. Mott]{Josh Mott} 
\address[J. Mott]{Department of Mathematics, University of North Carolina, Chapel Hill, NC, 27599}
\email{jrmott@unc.edu}
\author[T. Van Hoose]{Tim Van Hoose} \address[T. Van Hoose]{Department of Mathematics, University of North Carolina, Chapel Hill, NC, 27599}
\email{tvh@unc.edu}
\begin{document}
\begin{abstract}
    We establish a natural generalization of the H\"ormander condition for the multiplication of Besov distributions using the global wavefront set framework of Coriasco, Johansson, and Toft. Using this condition, we relate the microlocal perspective on multiplication of distributions to the classical paraproduct decomposition of pointwise multiplication arising in the paradifferential calculus of Bony.
\end{abstract}
\maketitle
\section{Introduction}

     In the context of low regularity nonlinear PDEs, a natural question which arises is whether one can define a product of two irregular objects. The classical Schwartz impossibility theorem provides an obstruction: there is no multiplication structure on the space of Schwartz distributions that naturally extends the product of continuous functions \cite{schwartzLimpossibiliteMultiplicationDistributions1954}. Many resolutions have been proposed to overcome this challenge, including H\"ormander's development of wavefront sets for tracking singularities of distributions \cite{hormanderAnalysisLinearPartial2007} and Bony's paradifferential calculus \cite{bonyCalculSymboliquePropagation1981}.
    
    H\"ormander's condition, which is a wavefront set statement precluding singularities of the two factors in opposite directions, allows one to define a natural product on distributions using no information about global structure, at the cost of the product only existing as a distribution. For many applications, one would like more structure on the resulting product, including that regularity and decay at infinity are preserved under the product.
    
    Bony's paradifferential calculus extends this multiplication to the scale of Besov spaces $B^\alpha_{p,q}$, function spaces that overlap with H\"older continuous functions $C^{\alpha}$ if $\alpha \notin \Z$ and $p = q = \infty$. Multiplication of Besov distributions can be performed using Littlewood-Paley theory, decomposing the standard product into three components
    \begin{equation*}
        fg = \sum_{N_1 \ll N_2} P_{N_1}f P_{N_2}g + \sum_{N_1 \approx N_2} P_{N_1} f P_{N_2}g + \sum_{N_2 \ll N_1} P_{N_1} f P_{N_2}g. 
    \end{equation*}
    The outer sums are the \emph{paraproducts}, and the middle term is the \emph{resonant} product. While the two paraproducts are always defined as tempered distributions, this is only true for the resonant product if the H\"older-Besov regularities of $f, g$ add to a positive number; in particular, the resonant product is the primary obstruction to defining $fg$.
    
    Recent work by \cite{dappiaggi_besov_2023, dappiaggiMicrolocalVersionYoungs2024} has explored the connections between the microlocal perspective of H\"ormander and the classical Young multiplication theorem, which gives sharp conditions for continuously multiplying Besov distributions. There, one may use a direct analogy with H\"ormander's work on $C^\infty$ singularities, at the cost of only guaranteeing  membership of the product in a local Besov space.
    
    In this paper, we extend the Besov H\"ormander condition by adopting the global wavefront set perspective of Coriasco, Johansson and Toft \cite{coriasco_global_2013}, allowing us to handle local singularities as well as the generic growth at infinity permitted in global Besov classes. Our primary result is a set of pointwise singularity and directional growth conditions which allow us to multiply global Besov distributions. In particular, our result allows trading off regularity between the two factors at different points and directions in phase space, enabling control over each factor to descend to global control of the product. Our main theorem is the following.
    \begin{theorem}\label[theorem]{T:maintheorem}
        Let $f, g \in \Sw'$ and suppose that the following two conditions hold:
        \begin{enumerate}
            \item For every $(x, \hat{\xi}) \in \R^d \times S^{d-1}$, there exist $\alpha, \beta \in \R$ with $\alpha+\beta > 0$ and so that 
            \begin{equation}\label{E:H1} 
                (x, \hat{\xi}) \notin \cWF^\alpha(f) \qtq{and} (x, -\hat{\xi})\notin \cWF^\beta(g).\tag{H1}
            \end{equation}
            \item For every $(\hat{x}, \hat{\xi}) \in S^{d-1} \times S^{d-1}$, there exist $\alpha, \beta \in \R$ with $\alpha+\beta > 0$ and so that 
            \begin{equation}\label{E:H2}
                (\hat{x}, \hat{\xi}) \notin \oscWF^\alpha(f) \qtq{and} (\hat{x}, -\hat{\xi})\notin \oscWF^\beta(g). \tag{H2}
            \end{equation}
        \end{enumerate}
        Then the following hold: 
        \begin{enumerate}
            \item $\Pi(f,g)$ exists as a tempered distribution on $\R^d$. 
            \item The pointwise product $fg$ is globally well-defined as a tempered distribution, and is given by the formula 
            \begin{equation*}
                fg = \Delta^\ast(f\otimes g),
            \end{equation*}
            where $\Delta \: \R^d \to \R^d \times\R^d$ is the diagonal map.
            \item Independently of the hypotheses above, $\sP_f(g)$ exists as a tempered distribution on $\R^d$.
        \end{enumerate}
        In addition, if we take any $(x, x, \xi, -\xi)$ which is normal to the diagonal, with \[(x, \xi) \notin \scWF^\alpha(f), (x, -\xi) \notin \scWF^\beta(g),\] then $\alpha+\beta > 0$ is a sharp condition for the resonant product to be well-defined for $f,g$ as above.    
    \end{theorem}
        \begin{remark}
            First, we notice that the theorem naturally permits us to have varying H\"older-Besov exponents throughout the wavefront sets of $f, g$, which significantly generalizes the classical paraproduct construction where this would be largely undetectable. 

            Second, this result generalizes the classical Young multiplication theorem, and in fact shows that existence of the product $fg$ as a distribution is equivalent to the existence of $\Pi(f,g)$ as a distribution, giving credence to the claim above that the obstruction to multiplying $f$ and $g$ is contained in the resonant product. 

            The condition $\alpha +\beta > 0$ is sharp in the sense that $\Pi \:\Sw(\R^d) \times \Sw(\R^d) \to \Sw'(\R^d)$ does not otherwise admit a continuous extension to the domain $B^\alpha(\R^d)\times B^{\beta}(\R^d)$.
        \end{remark}

    The remainder of the paper is organized as follows: in \Cref{S:Notation} we establish the notational conventions for harmonic and microlocal analysis for the rest of the paper. In \Cref{S:Properties}, we establish some fine properties relating to pullbacks and tensor products for the Besov wavefront sets. In \Cref{S:paraproducts} we prove \Cref{T:maintheorem}, and in \Cref{S:Appendix} we collect some short technical results relating to cutoffs.

\section*{Acknowledgements}
We thank Andy Badea for helpful discussions. J. Mott was supported by NSF RTG DMS-2135998. T. Van Hoose was supported by NSF Applied Math Grant DMS-2307384 and a Simons Dissertation Fellowship. 
\paragraph{\textbf{Statement on Generative AI Use}} Generative AI was not used at any stage of preparing this work. All errors, mathematical or otherwise, are the sole responsibility of the authors.
\section{Notation}\label{S:Notation}
     In this section, we introduce some notation that will be used throughout the remainder of the paper. We write $A \lesssim B$ or $B \gtrsim A$ to denote the inequality $A \leq CB$ for some constant $C > 0$, where $C$ may depend on parameters like the dimension or the indices of function spaces. If $A \lesssim B $ and $B \lesssim A$ both hold, then we write $A \sim B$. We will also denote generic dyadic integers by $N \in 2^{\N_0}$.
     We will also write $T^\ast\R^d \cong \R^{d} \times \R^d$ for the cotangent bundle of $\R^d$ and $T^\ast \R^d \setminus o$ for the cotangent bundle minus the zero section. We will refer to points generically in the cotangent bundle as $(x, \xi)$. We will often take the quotient of the bundle $T^\ast \R^d \setminus o$ by positive scaling in the fiber variable, yielding the cosphere bundle $S^\ast \R^d$. We perform this procedure frequently without comment, especially when dealing with the different wavefront sets below. Finally, we use $a \wedge b$ to denote the minimum of real numbers $a, b$ and $a \vee b$ for their maximum.
     
    \subsection{Harmonic Analysis, Littlewood-Paley theory and Besov spaces}
        We define the Fourier transform $\F$ by 
        \begin{equation*}
            (\F f)(\xi) = \hat{f}(\xi) = \int_{\R^d} e^{-ix\cdot \xi} f(x) \dx,
        \end{equation*}
        with inverse 
        \begin{equation*}
            (\F^{-1}g)(x) = \check{g}(x) =  (2\pi)^{-d} \int_{\R^d} e^{ix \cdot \xi} g(\xi) \d\xi.
        \end{equation*}
        We also fix a standard ball-annulus Littlewood-Paley partition of unity $(\vphi_N)_{N \in 2^{\N_0}}$ (cf \cite{bahouri_fourier_2011}). We define the Fourier multiplier operator $P_N$ by 
        $\widehat{P_N f} = \vphi_N f$ for every $N \in 2^{\N_0}$ and the convolution kernels $K_N$ by $K_N = \F^{-1}(\vphi_N)$.
        We define relations between frequency scales by 
        \begin{equation*}
            M \ll N \iff M \leq \tfrac{N}{8}, \quad M \approx N \iff \tfrac{N}{8} < M < 8N, \ \qtq{and}   M \gg N \iff M \geq 8N.
        \end{equation*} 
        \begin{definition}[Paraproduct Operators]\label[definition]{D:paraproducts}
            Let $f, g \in \Sw(\R^d)$. We define the low-high and resonant paraproducts by 
            \begin{equation*}
                \sP_f(g) = \sum_{N \ll M} P_N f P_M g \qtq{and} \Pi(f, g) = \sum_{M \approx N} P_M f P_N g.
            \end{equation*}
            This gives a decomposition (for $f, g \in \Sw(\R^d)$) of the product $fg$ into 
            \begin{equation*}
                fg = \sP_f(g) + \Pi(f,g) + \sP_g(f).
            \end{equation*}
            Next let $F(z_1, z_2) \: \R^d \times\R^d \to \Sw'(\R^d)$. We can define $2$-parameter analogues of the operators above by 
            \begin{equation*}
                \sP^{(2)}[F(z_1,z_2)](x,y) = \sum_{N \ll M} \int K_{N}(x-z_1) K_M(y-z_2) F(z_1, z_2) \dz_1 \dz_2
            \end{equation*}
            and 
            \begin{equation*}
                \Pi^{(2)}[F(z_1, z_2)](x,y) = \sum_{N  \approx M} \int K_N(x-z_1) K_M(y-z_2) F(z_1, z_2)\dz_1 \dz_2,
            \end{equation*}
        \end{definition}
        \begin{remark}
            We recover $\sP_f(g)$ and $\Pi(f,g)$ by applying $\sP^{(2)}$ and $\Pi^{(2)}$ to $f \otimes g$ and then pulling back by the diagonal map $\Delta(x) =(x,x)$.
        \end{remark}
        We also define the Besov spaces $B_{p,q}^s(\R^d)$ for $s \in \R, (p, q) \in [1, \infty]^2$ to be the space of Schwartz distributions with finite $B_{p,q}^s(\R^d)$-norm, defined by 
        \begin{equation*}
            \|u\|_{B_{p, q}^s} = \left\| N^s \|P_N u\|_{L^p} \right\|_{\ell_N^q},
        \end{equation*}
        with the usual modification to a supremum in the event that $p, q= \infty$.  For convenience we will write $B^s = B_{\infty, \infty}^s$.

        We state the usual paraproduct estimates as a lemma (cf. \cite{bahouri_fourier_2011}):
        \begin{lemma}\label[lemma]{L:Paraproduct estimates}
            Let $s_1, s_2 \in \R$. Then we have the following estimates for the paraproduct operators defined above:
            \begin{align*}
                \| \sP_{f}(g)\|_{B^{s_1 \wedge 0 + s_2}} &\lesssim \|f\|_{B^{s_1 \wedge 0}} \|g\|_{B^{s_2}},\\
                \|\Pi(f, g)\|_{B^{s_1 + s_2}} &\lesssim \|f\|_{B^{s_1}}\|g\|_{B^{s_2}} \qtq{as long as $s_1 + s_2 > 0$.}
            \end{align*}
        \end{lemma}
        
      \subsection{Microlocal Analysis} 
        We next establish our conventions for the necessary tools from microlocal analysis. First, we define the symbol class $\scS_{r_1, r_2}^{m_1, m_2}$. This is a slight generalization of the classical scattering symbol class, allowing for more flexibility in the symbolic conversion between differentiation and decay in the variable/covariable.
         \begin{definition}
            Fix $m_1, m_2, r, \rho \in \R$ and let $a(x, \xi)$ be a smooth function on $\R^{2d}$. We say that $a \in \scS_{r, \rho}^{m_1, m_2}(\R^{2d})$ if for every pair of multiindices $\alpha, \beta$ we have the bounds 
            \begin{equation*}
                |D_x^\alpha D_\xi^\beta a(x, \xi)| \lesssim \jbrak{x}^{m_1 -r|\alpha|} \jbrak{\xi}^{m_2 - \rho|\beta|}.
            \end{equation*}
            We write for convenience $\scS^{m_1, m_2}_{1,1} = \scS^{m_1, m_2}$.
        \end{definition}
        Next for any symbol $a(x, \xi) \in \scS_{r_1, r_2}^{m_1, m_2}$, we quantize to get an operator $\Op_t(a) : \Sw(\R^d) \to \Sw(\R^d)$ by
        \begin{equation*}
            \Op_t(a) = \int_{\R^d} a((1-t)x + ty, \xi)e^{i(x-y)\cdot\xi}\dy\d\xi,
        \end{equation*}
        for any $t \in [0,1]$. In the case $t =0$, we see that this is the usual Kohn-Nirenberg quantization, and we write 
        \begin{equation}
            \Op(a)u = a(x, D)u(x).
        \end{equation}
        Following \cite{coriasco_global_2013}, we restrict our attention to spaces of distributions for which quantization is well behaved.
        \begin{definition}
            Let $B$ be a Fr\'echet space so that for some $d \geq 1$,
            \begin{equation*}
                \Sw(\R^d) \subseteq B \subseteq \Sw'(\R^d).
            \end{equation*}
            Fix $r, \rho \in [0, 1]$. Then we say that $B$ is \emph{$\scS$-admissible} with parameters $r, \rho, d$ if 
                $\Op_t(a) \: B \to B$ continuously
            for any $a \in \scS_{r, \rho}^{0, 0}(\R^{2d})$ and $t \in \R$.
        \end{definition}
      \begin{proposition}\label[proposition]{P:scadmissible}
            Let $s \in \R$ and $d\geq 1$. Then $B^s(\R^d)$ is $\scS$-admissible with parameters $(1, 1, d)$. 
        \end{proposition}
        \begin{proof}
            The first thing to note is $\scS_{1,1}^{0,0} \subseteq \scS_{0, 1}^{0, 0}$ where the latter space is the classical H\"ormander symbol class. By using Fourier inversion we can rewrite 
            \begin{equation*}
                \Op_t(a) = \Op(b(t, x, \xi))
            \end{equation*}
            for some explicit symbol $b$  which is a symbol in $\scS_{0, 1}^{0,0}$ (uniformly in $t$). We then appeal to \cite[Theorem 1]{bourdaudLpEstimatesCertainNonregular1982}, which immediately yields the boundedness on $B_{\infty, \infty}^s$ for any $s \in \R$.
        \end{proof}
        In this framework, the paraproducts $\sP^{(2)}$ and $\Pi^{(2)}$ are zeroth-order $\Psi$DOs.
        \begin{proposition}\label[proposition]{P:micro_paraproducts}
            The symbols of $\sP^{(2)}$ and $\Pi^{(2)}$ viewed as operators on $2$-variable distributions, are in the symbol class 
            \begin{equation*}
                \scS^{0,0}_{1, 1}(\R^{4d}).
            \end{equation*}
        \end{proposition}
        \begin{proof}
            This is straightforward; we may explicitly calculate the symbols, which yields
            \begin{equation*}
                \sigma(\sP^{(2)}) (x, y, \xi, \eta) = \sum_{M \ll N} \varphi_M(\xi) \varphi_N(\eta)
            \end{equation*}
            and 
            \begin{equation*}
                \sigma(\Pi^{(2)})(x,y, \xi, \eta) = \sum_{M \approx N} \varphi_M(\xi) \varphi_N(\eta).
            \end{equation*}
            From here, it is easy to see the symbolic bounds 
            \begin{equation*}
                |\partial_x^{\alpha_1} \partial_y^{\alpha_2} \partial_\xi^{\beta_1} \partial_\eta^{\beta_2} \sigma(T)(x, y, \xi, \eta)| \lesssim \jbrak{x+y+ \xi+ \eta}^{-(|\alpha_1| + |\alpha_2| + |\beta_1|+|\beta_2|)},
            \end{equation*}
            where $T \in \{\sP^{(2)}, \Pi^{(2)}\}$. This implies the desired result.
        \end{proof}
        We also define two classes of cutoffs that will be used to localize distributions to different portions of the cotangent bundle.
        \begin{definition}\label[definition]{D:localization}
            Let $U \subset \R^d$ be open and $\Gamma \subset \R^d \setminus \{0\}$ be an open cone. Fix $(x_0,\xi_0) \in U \times \Gamma$. 
            \begin{enumerate}
                    \item A function $\phi \: \R^d \to \R$ is called a \emph{cutoff} to $x_0$ if $\phi \in C_c^\infty(U)$, $0 \leq \phi(\cdot) \leq 1$ and if there is an open neighborhood of $x_0$ where $\phi \equiv 1$. 
                \item A function $\psi \: \R^d \to \R$ is a \emph{directional cutoff to $\xi_0$ (on $\Gamma$)} if there exists $R >0 $ and an open conical neighborhood $\xi_0 \in \Gamma_1 \subset \Gamma$ so that the following hold:
                \begin{itemize}
                    \item $0 \leq \psi(\cdot) \leq 1$ and $\rm{supp}(\psi) \subset \Gamma$;
                    \item For $\xi \geq R$ and $t \geq 1$, $\psi(t\xi) = \psi(\xi)$, and
                    \item $\psi(\xi) = 1$ on $\Gamma_1$ and for $|\xi| \geq R$. 
                \end{itemize}
            \end{enumerate}
            For brevity, we will typically omit the point and open set when clear from context. 
            We denote the set of cutoffs to $x_0 \in U$ by $\D_{x_0}(U)$ and the set of directional cutoffs to $(\xi_0, \Gamma)$ by $\D_{\xi_0}^{\rm{dir}}(\Gamma)$.
            We also define for $x_0 \in \R^d$ a \emph{fattened cutoff} to $x_0$ to be a pair $(\vphi, \bbar{\vphi})$ so that $\vphi$ is a cutoff to $x_0$, and $\bbar{\vphi}$ is a cutoff to a neigborhood of $\supp{\vphi}$. 
        \end{definition}

        We next define our scattering wavefront sets. Recall that for functions $f, g \: \R^d \to \R$, we form their tensor product $f \otimes g \: \R^{2d} \to \R$ by $(f \otimes g)(x,\xi) = f(x)g(\xi)$. Additionally, if $A, A' \subset T^\ast \R^d \setminus o$ we define
        \begin{equation*}
            A \boxtimes A' = \{(x,y, \xi, \eta)\mid (x, \xi) \in A, (y, \eta) \in A'\}.
        \end{equation*}
        We now define 3 different wavefront sets components to capture membership in $B$:
        \begin{definition}\label{D:wfsets}
            Let $B$ be a Fr\'echet space so that $\Sw(\R^d) \subseteq B \subseteq \Sw'(\R^d)$  for some $d$.
            Let $U \subset \R^d$ open and $\Gamma, \Gamma_1, \Gamma_2$ be open cones. Let $u \: \R^d \to \R$. Then 
            \begin{enumerate}
                \item $(x_0, \xi_0) \notin \cWF(u; B)$ if there exists $\phi \in \D_{x_0}(U)$ and $\psi \in \Ddir_{\xi_0}(\Gamma)$ so that $\Op(\phi \otimes \psi)u \in B$. 
                \item $(x_0, \xi_0) \notin \FWF(u; B)$ if there exists $\phi \in \D_{\xi_0}(U)$ and $\psi \in \Ddir_{x_0}(\Gamma)$ so that $\Op(\psi \otimes \phi)u \in B$. 
                \item $(x_0, \xi_0) \notin \oscWF(u; B)$ if there exist $\psi_1, \psi_2 \in \Ddir_{x_0}(\Gamma_1) \times \Ddir_{\xi_0}(\Gamma_2)$ so that $\Op(\psi_1 \otimes \psi_2)u \in B$.
            \end{enumerate}
            
            The \emph{scattering $B$-wavefront set of $u$}, denoted $\scWF(u; B)$ is the union 
            \begin{equation*}
                \scWF(u; B) = \cWF(u; B) \cup \FWF(u; B) \cup \oscWF(u; B).
            \end{equation*}
            In the case $B = \Sw(\R^d)$, this is the usual scattering wavefront set. Our primary concern will be the case $B = B^\alpha(\R^d)$, in which case we will write $\scWF(u; B^\alpha) := \scWF^\alpha(u)$, with the obvious notational modifications to the other wavefront sets.
        \end{definition} 
        \begin{remark}
            We call the set $\cWF$ the \emph{classical} wavefront set, and note that in the Schwartz case having empty $\cWF$ is equivalent to being smooth. The sets $\FWF$ and $\oscWF$ are referred to as the \emph{Fourier} and \emph{oscillating} wavefront sets.
        \end{remark}
        \begin{remark}\label[remark]{R:settheory}
            We have the containment 
            \begin{equation*}
               \scWF(u; \Sw') \subseteq \scWF(u; B) \subseteq \scWF(u; \Sw),
            \end{equation*}
            which follows from the containment $\Sw \subseteq B \subseteq \Sw'$.
        \end{remark}
        \begin{remark}
            We note that $\FWF^s(f) = \FWF^t(f)$ for all $s, t \in \R$. To see this, note that if $(x_0, \xi_0) \notin \FWF^s(f)$, then there exist $\psi_{x_0} \in \Ddir_{x_0}$ and $\vphi_{\xi_0} \in \D_{\xi_0}$ so that $\psi_{x_0}\vphi_{\xi_0}(D_x) f \in B^s$. Now choose slightly smaller cutoffs $\bbar{\psi}_{x_0}, \bbar{\vphi}_{\xi_0}$ according to \Cref{L:sepsupps} so that $\bbar{\vphi}_{\xi_0}\vphi_{\xi_0} = \bbar{\vphi}_{\xi_0}$ and $\bbar{\psi}_{x_0} \psi_{x_0} = \bbar{\psi}_{x_0}$. We now observe the identity 
            \begin{equation*}
                \bbar{\psi}_{x_0} \bbar{\vphi}_{\xi_0}(D)f = \bbar{\psi}_{x_0} \bbar{\vphi}_{\xi_0}(D) \psi_{x_0} \vphi_{\xi_0}(D)f + \bbar{\psi}_{x_0} \bbar{\vphi}_{\xi_0}(D)(1-\psi_{x_0})\vphi_{\xi_0}(D) f.
            \end{equation*}
            The second term is Schwartz by \Cref{L:sepsupps}, and the first is Schwartz because $\bbar{\psi}$ is a zeroth order multiplier and $\bbar{\vphi}_{\xi_0}$ has compact Fourier support.
        \end{remark}
        We will also need the following enlargements of the wavefront sets:
        \begin{definition}
           First, for the classical wavefront set, we define 
            \begin{equation}
                \cWF_0^s(f) = \cWF^s(f) \cup (\supp{f} \times \{0\}).
            \end{equation}
            For the Fourier and oscillating wavefront sets,  define $\pi_1, \pi_2 \: \R^{2d} \to \R^d$ to be the canonical projections in the physical and fiber variables. Then we set 
            \begin{equation}
                \FWFt^s(f) = \FWF^s(f) \cup (\{0\} \times \pi_2 \FWF^s(f)).
            \end{equation}
            and
            \begin{equation}
                \begin{split}  
                \oscWFt^s(f) = &\oscWF^s (f) \cup \left[\{0\} \times \pi_2\left(\cWF^s(f) \cup \oscWF^s(f)\right) \right] \cup \\ 
                &\left[\pi_1 (\FWF^s(f) \cup \oscWF^s(f)) \times \{0\} \right] \cup \{(0,0)\}
                \end{split}
            \end{equation}
        \end{definition}
        \begin{remark}
            These definitions will be used in stating the natural version of the tensor product theorem for wavefront sets in the Besov regime; we will give a brief indication of why the enlargements are necessary here. 

            The essential reason for the enlargement is because the natural space in which to embed the Fourier and oscillating wavefront sets of $f \otimes g$ are of the form 
            \begin{equation*}
                \left\{(a \hat{x}, b \hat{y}, \xi, \eta)) \resizemid \begin{aligned}
                    &(\hat{x}, \xi) \in S^{n-1} \times \R^n \\
                    &(\hat{y}, \eta) \in S^{m-1} \times \R^m \\
                    &a^2+b^2=1
                \end{aligned} \right\} \qtq{for} \FWF(f \otimes g)
            \end{equation*}
            and similarly
            \begin{equation*}
                \left\{(a \hat{x}, b \hat{y}, \alpha \hat{\xi}, d \hat{\eta})) \resizemid \begin{aligned}
                    &(\hat{x}, \hat{\xi}) \in S^{n-1} \times S^{n-1} \\
                    &(\hat{y}, \hat{\eta}) \in S^{m-1} \times S^{m-1} \\
                    &a^2+b^2= \alpha^2+\beta^2=1
                \end{aligned} \right\} \qtq{for} \oscWF(f \otimes g)
            \end{equation*} 
            in particular, it is possible to take $a = 0, b=1$ (or various combinations thereof, in the case of the oscillating wavefront set). In the proofs, the enlargements defined above allow us to compactly compensate for the loss of information coming from attempting to take cutoffs in the `zero direction'. More precisely, we note that
            \begin{equation*}
                (0, \xi) \notin \FWFt(f) \iff \exists \vphi \in \mathscr{D}_\xi(\R^d) \text{ so that } \vphi(D_x) f \in B^s,
            \end{equation*}
            with similar versions for the oscillating wavefront set. 
        \end{remark}
        The consequence of \Cref{P:scadmissible} and the definition above is the applicability of the following theorem in \cite{coriasco_global_2013}:
        \begin{theorem}[Theorem 2.6, \cite{coriasco_global_2013}]
            Let $B$ be an $\scS$-admissible space. Then we have the following characterization of `$B$-regularity': for any $f \in\Sw'(\R^d)$, 
            \begin{equation*}
                f \in B \iff \scWF(f;B) = \emptyset.
            \end{equation*}
        \end{theorem}
        We next state some facts about elliptic regularity and microlocalization in this scale of spaces. The main result we will use is \cite[Theorem 3.1]{coriasco_global_2013}, whose result we restate here, albeit in slightly different notation and language. We begin by defining the relevant notion of characteristic sets.
        \begin{definition}
            Let $a = a(x,\xi) \in \scS_{r, \rho}^{0,0}(\R^{2d})$, $r, \rho \geq 0$. We say that 
            \begin{enumerate}
                \item $a$ is \emph{$\mathrm{c}$-invertible} at $(x_0, \xi_0) \in \R^d \times (\R^d \setminus \{0\})$ if there exists an open neighborhood $U$ of $x_0$, an open conical neighborhood $\Gamma$ of $\xi_0$, and $R, c > 0$ so that $|a(x, \xi)| > c$ for $x \in U, \xi \in \Gamma \cap \{|\xi| \geq R\}$.
                \item $a$ is \emph{$\F$-invertible} at $(x_0, \xi_0) \in (\R^d \setminus \{0\}) \times \R^d$ if there exists an open neighborhood $U$ of $\xi_0$, an open conical neighborhood $\Gamma$ of $x_0$ and $R, c > 0$ so that $|a(x, \xi)| > c$ for $x \in \Gamma \cap \{ |x| \geq R\}$ and $\xi \in U$.
                \item Finally, $a$ is \emph{$\mathrm{osc}$-invertible} at $(x_0, \xi_0) \in (\R^d \setminus \{0\})^2$ if there exist open cones $\Gamma_1 \ni x_0$, $\Gamma_2 \ni \xi_0$ and constants $R, c > 0$ so that $|a(x, \xi)| > c$ on $(\Gamma_1 \cap \{|x| \geq R\}) \times (\Gamma_2 \cap \{|\xi| \geq R\})$. 
            \end{enumerate}
            For each $\tau \in \{ \mathrm{c}, \F, \mathrm{osc}\}$ the complement of the points where $a$ is $\tau$-invertible is the \emph{$\tau$-characteristic set}, written $\mathrm{Char}^\tau(a)$.
        \end{definition}

        \begin{remark}
            In the $\rm{c}$-invertible case, the fact that we demand invertibility of the total symbol rather than the principal symbol does not change the resulting wavefront set, though the characteristic set is smaller than the classical one, cf. \cite[Remark 1.22]{coriasco_global_2013}.
        \end{remark}
        
        We conclude this section by recalling the Besov wavefront set equivalent of microlocality and microellipticity.
        \begin{lemma}[Theorem 3.1, \cite{coriasco_global_2013}] \label[lemma]{L:microlocality_and_microellipticity}
            Let $s \in \R$, $f \in \Sw'$, and take any symbol $a \in \scS_{1, 1}^{0,0}(\R^{2d})$. Then for $\tau \in \{c, \F, \mathrm{osc}\}$,
            \begin{equation*}
                \prescript{\tau}{}\WF^s(\Op(a)f) \subset \prescript{\tau}{}\WF^s(f) \subset \prescript{\tau}{}\WF^s(\Op(a)f) \cup \Char^\tau(a).
            \end{equation*}
        \end{lemma}

        \section{Fine Properties of the Besov Scattering Calculus}\label{S:Properties} 

            In this section, we examine the behavior of Besov wavefront sets under pullback and tensor products, which we then use in the sequel to prove the main result.

        \subsection{Wavefront Sets Under Pullback}
        We prove several facts about the transformation laws of general scattering wavefront sets as an analogy to classical results about smooth wavefront sets, before restricting our attention back to Besov spaces. Unlike in the smooth setting, we require control over our diffeomorphisms at infinity to handle pullbacks. We begin by defining a modification of \cite[Definition 16]{coriascoFIO1999}, which allows us to pull back points at infinity in our compactified physical or fiber variables.
        \begin{definition}
            We say that $f : \R^d \rightarrow \R^d$ is a classical SG-compatible diffeomorphism if the following hold:
            \begin{enumerate}
                \item $f$ is a diffeomorphism between open subsets $U$ and $V$ in $\R^d$ such that
            \begin{align}
                \forall x & \in U, \quad |\partial^\alpha_x f(x)| \lesssim \jbrak{x}^{1-|\alpha|} \\
                \forall y & \in V, \quad |\partial^\alpha_y (f^{-1})(y)| \lesssim \jbrak{y}^{1-|\alpha|},
            \end{align}
            \item There exist a $U' \subset U$, $V' \subset V$, and a $\delta > 0$ such that $f\big|_{U}$ is a diffeomorphism from $U' \rightarrow V'$ and
            \begin{align}
                \forall x \in U, &\quad B(x, \delta\jbrak{x}) \subset U',\\
                \forall y \in V, & \quad B(y, \delta \jbrak{y}) \subset V'.
            \end{align}
            \item $f$ has an asymptotic expansion into homogeneous terms in $x$, i.e. for $x$ outside of some compact set
            \begin{align}
                f(x) \sim \sum_j f_{1-j}(x)
            \end{align}
            for $f_{1-j}$ homogeneous in $x$ of degree $1-j$, modulo an error term $r(x)$ satisfying $|\partial_x^\alpha r| \lesssim \jbrak{x}^{-N}$ for any $N$.
            \end{enumerate}
        \end{definition}
        We next prove that scattering wavefront sets are invariant under pullback by a classical SG-compatible diffeomorphism.
        \begin{proposition}
            \label[proposition]{P:diffeo-invariance}
            Let $f$ be a classical $SG$ diffeomorphism and take $u \in \D'$. For any SG-compatible space $B$ such that $f^*(B) = B$, we have that
            \begin{align}
                \scWF(f^*u) = f^* \left(\scWF(u)\right).
            \end{align}
        \end{proposition}
        \begin{proof}
            Suppose first that $(x_0, \xi_0) \notin \FWF^s(f^*u)$. Thus for some symbol $a \in \scS^{0,0}(\R^{2d})$ that is invertible on some conical exterior $\Gamma \cap B_R^c \ni x_0$ and some open set $U \ni \xi_0$ such that $\Op(a) (f^*u) \in B$. By \Cref{L:further-localization}, we may without loss of generality take $a$ to be an $x$-polyhomogeneous classical SG symbol (c.f. \cite{egorov_pseudo-differential_1997} or \cite{coriasco_fourier_2002}), meaning it possesses an asymptotic expansion in homogeneous $x$ terms $a(x, \xi) \sim \sum_j a_{-j}(x, \xi)$. Note that $|a(x, \xi)| > c > 0$ on $\Gamma \cap B_R^c \times U$ implies that on some smaller sets $\Gamma, U$ with larger $R$, $|a_0(x, \xi)| > c' > 0$.
            
            We define $\varphi(x, \xi) = f(x) \cdot \xi$, which satisfies the requirements to be a SG phase function. Then \cite[Theorem 7]{coriascoCalculusFourierIntegral2016} tells us that the composition $\Op(a) (f^*u)$ can be written as an SG Fourier Integral Operator $\Op_\varphi(h)$ with phase function $\varphi$ and amplitude
            \begin{align}
                \label{E:FIO-amplitude-expansion}
                h(x, \eta) \sim \sum_{\alpha \in \N^d} \frac{1}{\alpha!} (\partial^\alpha_\xi a)(x, f^*(\eta)) D^\alpha_{x'} \left[ e^{i\psi(x, x', \eta)} \right]\bigg|_{x'=x},
            \end{align}
            where $f^*(\eta) = d_x \varphi(x, \eta)$ and $\psi$ is the first order Taylor remainder
            \begin{align}
                \psi(x, x', \eta) = (f(x') - f(x)) \cdot \eta - (x'-x) \cdot f^*(\eta).
            \end{align}
            We next see that $\Op_\varphi(h)$ is given by the pullback of a pseudodifferential operator, since
            \begin{align}
                \Op_\varphi(h(x, \eta)) u & = \frac{1}{(2\pi)^d} \iint e^{i(f(x) - x') \cdot \eta} h(x, \eta) u(x') \, dx' \\
                & = f^* \left( \Op(h(f^{-1}(y), \eta) u(y) \right).
            \end{align}
            Define $b(y, \eta) = h(f^{-1}(y), \eta)$. From \eqref{E:FIO-amplitude-expansion}, we know $b$ is in $\scS^{0,0}$. Since $f^*(B) = B$, up to a Schwartz error,
            \begin{align}
                \Op(b(y, \eta)) u = (f^{-1})^* \Op_\varphi(h(x, \xi)) u = (f^{-1})^* \Op(a(x, \xi)) (f^*u) \in B.
            \end{align}
            To conclude, we show that $b(y, \eta)$ is $\mathcal{F}$-microelliptic near $(f^{-1})^*(x_0, \xi_0)$. We use the symbol expansion \cref{E:FIO-amplitude-expansion} to derive this from the properties of $a$.

            Examining the $x'$-derivative term, we see that the Taylor remainder $\psi$ has vanishing first derivative. Thus, summands of \eqref{E:FIO-amplitude-expansion} for $\alpha \neq 0$ are given by terms of the form
            \begin{align}
                \frac{1}{\alpha!} (\partial_\xi^\alpha a)(x, f^*(\eta)) \prod_{j}(\partial^{\alpha_j}_x (f(x) \cdot \eta)),
            \end{align}
            for some multiindices $\alpha_j$ such that $\alpha = \sum_j \alpha_j$ and $|\alpha_j| \geq 2$ for each $j$. By our hypotheses on $a$ and $f$, this tells us that \eqref{E:FIO-amplitude-expansion} can be written as an asymptotic expansion of homogeneous symbols in $x$. Therefore, we may apply the SG Borel-type lemma from \cite{coriasco_asymptotic_2014} to conclude that
                $h(x, \eta) - a(x, f^*(\eta)) \in SG^{-1, 0}_{1,1}$
            and therefore that $|h(x, \eta) - a(x, f^*(\eta))| \lesssim \jbrak{x}^{-1}$. Then $|a(x, \xi)| > c'$ on $\Gamma \cap B^c_R \times U$ implies that $h(x, \xi) > c'/2$ on $\Gamma \cap B^c_{R'} \times V$  for some smaller $V \subset U$ and $R'$ sufficiently large.\footnote{The pullback $f^*(U)$ depends on the spatial variable $x$. However, the asymptotic linearity of $f$ ensures that there always exists an open $V \subset \R^d$ such that $V \subset f^*_x(U)$ for all $x \in \Gamma' \cap B^c_{R'}$}. Thus $b(y, \eta)$ is $\mathcal{F}$-microelliptic near $(f^{-1})^*(x_0, \xi_0)$ and $(x_0, \xi_0) \notin f^*\WF(u)$. The opposite inclusion follows from considering $f^{-1}$ and performing the same analysis.

            For the oscillating wavefront set, one picks $a$ instead to be simultaneously homogeneous in $x$ and $\xi$ before applying the same techniques. The statement is also known to already hold in more generality for the classical wavefront set \cite{hormanderAnalysisLinearPartial2007}.
        \end{proof}

        To prove general transformation laws under pullback by a smooth embedding, we begin with the simpler case of a smooth inclusion $\iota : \R^d \rightarrow \R^{2d}$ with $\iota(x) = (x, 0)$ and then use diffeomorphism invariance.
        \begin{proposition}
            \label[proposition]{P:WF-trace-inclusion}
            Take $\iota : \R^d \rightarrow \R^{2d}$ with $\iota(x) = (x, 0)$ and any SG-compatible space $B(\R^{2d})$ such that $g^* B(\R^{2d}) \subset B(\R^d)$. Then for any $u \in \Sw'$ with both
            \begin{align}
                \cWF(u) \cap \{(x, 0, 0, \eta) : x \in \R^d, \eta \in \R^d \setminus 0 \} = \varnothing, \\
                \oscWF(u) \cap \{(x, 0, 0, \eta) : x \in \R^d \setminus 0, \eta \in \R^d \setminus 0 \} = \varnothing.
            \end{align}
            Then $\iota^*(u)$ is defined as a distribution and $\tauWF(\iota^* u) \subset \iota^* \left( \tauWF(u) \right)$ for $\tau \in \{ \mathrm{c}, \F, \mathrm{osc}\}$.
        \end{proposition}
        \begin{proof}
            Again, we prove this result only for the Fourier wavefront set component; the classical component again follows from standard results, while the oscillating requires only requires a minor modification of the compactness argument.
            
            Suppose that $(x_0, \xi_0) \notin \iota^* \FWF(u)$, which is equivalent to $(x_0, 0, \xi_0, \eta) \notin \FWF(u)$ for all $\eta \in \R^d$. Thus, for each $\eta \in \R^d$, we have $\psi_\eta \in \Ddir_{(x_0, 0)}$ and a $\varphi_\eta \in \D_{(\xi_0, \eta)}$ such that $\Op(\psi_\eta \otimes \varphi_\eta)u \in B(\R^{2d})$. Likewise, by the hypothesis on $\oscWF(u)$, we know that there is a $\Tilde{\psi} \in \Ddir_{(x_0, 0)}$ and a $\Tilde{\varphi} \in \Ddir_{(\xi_0, 0)}$ such that $\Op(\Tilde{\psi} \otimes \Tilde{\varphi}) u \in B(\R^{2d})$.
    
            We note that the $\eta$ such that $(\xi_0, \eta)$ is outside the set where $\Tilde{\psi} = 1$ is compact. Thus, we may pick finitely many $\eta_j$ in this range such that $\supp \varphi_{\eta_j}$ covers $\{(\xi, \eta) : \Tilde{\psi}(\xi, \eta) = 1\}$. Then we define $\psi = \Tilde{\psi} \cdot \prod_j \psi_{\eta_j} \in \Ddir_{(x_0, 0)}$ and $\varphi = \Tilde{\varphi} + \sum_j \varphi_\eta$, noting that by construction $\Op(\psi \otimes \varphi) u \in B(\R^{2d})$.
    
            For our particular choice of diffeomorphism, an explicit evaluation of $\Op(a(x, \xi)) \iota^* u$ for any symbol $a$ gives us that
            \begin{align}
                \Op(a(x, \xi)) \iota^* u = \iota^* \left( \Op(a^*(x, y, \xi, \eta) u \right),
            \end{align}
            for $a^*(x, y, \xi, \eta) = a(x, \xi)$. Then, picking $\chi(\xi) \in \D_{\xi_0}$ to be such that $\supp \chi(\xi)$ is properly contained in $\supp \varphi$, we define $a(x, \xi) = \psi(x, 0) \chi(\xi)$. Clearly $a$ is $\mathcal{F}$-microelliptic near $(x_0, \xi_0)$. Additionally, we note that
            \begin{align}
                \Op(\psi(x, 0) \chi(\xi)) u = \Op(\psi(x, y) \psi(x, 0) \chi(\xi)) u + \Op((1-\psi(x, y)) \psi(x, 0) \chi(\xi)) u.
            \end{align}
            The first of these terms is in $B(\R^{2d})$ by the fact that $\supp \psi(x,y)\psi(x, 0) \chi(\xi)$ is properly contained in $\supp \psi(x, y) \varphi(\xi, \eta)$ and \Cref{L:further-localization}. The second of these terms has support away from the range of $f$, so it pulls back to $0$. Thus,
            \begin{align}
                \Op(\psi(x, 0) \chi(\xi) (\iota^* u) & = \iota^*\left(\Op(\psi(x, 0) \chi(\xi)) u  \right) \\ & = \iota^*\left(\Op(\psi(x, y) \psi(x, 0) \chi(\xi)) u  \right) \in \iota^*(B(\R^{2d})) \subset B(\R^d).
            \end{align}
            Hence $(x_0, \xi_0) \notin \FWF(\iota^*u)$.
        \end{proof}
        We now prove our general result about pullbacks of Besov wavefront sets under compatible diffeomorphisms.
        \begin{theorem}\label[theorem]{T:pullback}
            Let $s > 0$ and let $g: \R^{2d} \rightarrow \R^{2d}$ be a strong $SG$ diffeomorphism. Set $f = g \circ \iota \: \R^d \to \R^{2d}$. Suppose that for some $u \in \Sw'$ we have 
            \begin{align}
            \label{eq:normal-bundle-hypothesis1}
                \cWF^s(u) \cap \{(g^{-1})^*(x, 0, 0, \eta) : x \in \R^d, \eta \in \R^d \setminus 0\} = \varnothing \\
            \label{eq:normal-bundle-hypothesis2}
                \oscWF^s(u) \cap \{(g^{-1})^*(x, 0, 0, \eta) : x \in \R^d \setminus 0, \eta \in \R^d \setminus 0\} = \varnothing
            \end{align}
            we have that $\tauWF^s(f^* u) \subset f^* \left( \tauWF^s(u) \right)$ for $\tau \in \{ \mathrm{c}, \F, \mathrm{osc}\}$ and all $s \in \R$.
        \end{theorem}
        \begin{proof}
            We first note that the condition of $g$ being a classical SG-diffeomorphism implies that it has bounded derivatives of all orders and has Jacobian bounded away from $0$. Thus, by \cite[Theorem 4.3]{triebelTheoryFunctionSpaces1992}, $g^*(B) = B$. Additionally, by repeated application of the Besov hyperplane trace theorem \cite[Theorem 3.2]{schneiderTraceOperatorsBesov2010}, we know that $\iota^*(B^s(\R^{2d})) \subset B^s(\R^d)$. Thus, the SG space hypotheses of \Cref{P:WF-trace-inclusion,P:diffeo-invariance} are satisfied.
            By \Cref{P:diffeo-invariance}, $\tauWF^s(g^* u) = g^* \tauWF^s(u)$. Then the hypotheses \eqref{eq:normal-bundle-hypothesis1} and \eqref{eq:normal-bundle-hypothesis2} exactly imply the corresponding assumptions in \Cref{P:WF-trace-inclusion}. Thus, $\WF^s (\iota^* g^* u) \subset \iota^* \WF^s(g^* u) = \iota^* g^* \WF^s(u) = f^* \WF^s(u)$.
        \end{proof}

        \begin{corollary}
            \label[corollary]{C:diagonal-pullback}
            For the diagonal map $\Delta(x) = (x, x)$, if we take $s > 0$ and some $u \in \Sw'(\R^{2d})$ such that
            \begin{align}
                \cWF^s(u) \cap \{(x, x, -\eta, \eta) : x \in \R^d, \eta \in \R^d \setminus \{0\}\} = \varnothing \label{E:int1} \\
                \oscWF^s(u) \cap \{(x, x, -\eta, \eta) : x \in \R^d \setminus \{0\}, \eta \in \R^d \setminus \{0\}\} = \varnothing, \label{E:int2}
            \end{align}
            then $\Delta^* u \in \Sw'(\R^d)$ is defined and
            \begin{align}
                \scWF^s(\Delta^* u) \subset \Delta^* \scWF^s(u).
            \end{align}
        \end{corollary}
        \begin{proof}
            This follows from taking $g$ to be the rotation $g(x, y) = (x-y, x+y)$.
        \end{proof}
        \begin{remark}
            Here we recognize the classical diagonal $\cNDelta$ as the second set in \eqref{E:int1}, and we define $\oscNDelta$ to be the second set in \eqref{E:int2}; we refer to this as the \emph{diagonal at infinity} in the sequel. We also note that no conditions on $\FWF$ ever need to be checked to pull back.
        \end{remark}
    \subsection{Tensor Products}
    We begin by defining the product of two distributions in the sense of H\"ormander:
    \begin{definition}
        Let $u, v \in \Sw'(\R^d)$, and let $\Delta \: \R^d \hookrightarrow \R^d \times \R^d$ be the diagonal map. We formally define  the product $uv$ by 
        \begin{equation*}
            uv = \Delta^\ast (u \otimes v).
        \end{equation*}
        Here, the asterisk denotes pullback, in the usual fashion.
        If $u, v$ are Schwartz functions, this coincides with the pointwise product.
    \end{definition}
    We turn to the tensor product property of the Besov wavefront sets, which is analogous to the classical result of \cite[Lemma 11.6.3]{hormanderLecturesNonlinearHyperbolic1997} for the Sobolev wavefront set. The local version of this result was first proved in \cite{dappiaggi_besov_2023}. To the best of our knowledge this global result is new.
    
    \begin{proposition}\label[proposition]{P:tensorprod}
        Let $f \in \Sw'(\R^n)$ and $g \in \Sw'(\R^m)$. Let $s, t\in \R$. Then for any $r < \min\{s, t, s+t\}$ and $\tau \in \{\rm{c}, \rm{osc}\}$ we have the following inclusions of $\tauWF(f\otimes g)$:
            \begin{equation}\label{E:cwftensor}
                \cWF^{r}(f \otimes g) \subset \cWF_0^{s}(f) \boxtimes \WF(g) \cup \WF(f) \boxtimes \cWF_0^{t}(g), \qtq{and}
            \end{equation}
            \begin{equation}\label{E:oscwftensor}
                \oscWF^{r}(f \otimes g) \subset \oscWFt^s(f) \boxtimes \R^{2m} \cup \R^{2n} \boxtimes \oscWFt^t(g).
            \end{equation}
        For the Fourier wavefront set, we have the following inclusion \emph{for any} $r, s, t$:
        \begin{equation}\label{E:FWFtensor}
            \FWF^r(f\otimes g) \subset \FWFt^s(f) \boxtimes \R^{2m} \cup \R^{2n} \boxtimes \FWFt^t(g).
        \end{equation}
    \end{proposition}
    \begin{proof}
        We begin by proving a Besov estimate for tensor products. To estimate the Besov norm of $f\otimes g$, we look at
        \begin{equation*}
            P_{N}^{x,y} \sum_{N_1, N_2} \left(P_{N_1}^x f \otimes P_{N_2}^y g\right)(x,y).
        \end{equation*}
        Here, $P_N^{x,y}$ has symbol $\varphi_{N}(\xi, \eta)$, which has support where $|\xi| + |\eta| \sim N$. By equivalence of norms in finite dimensions, this means that $\max\{|\xi|, |\eta|\} \sim N$ (after potentially changing the constant by a fixed factor). In particular, the sum is nonzero precisely when $\max\{ N_1, N_2\} \sim N$. But then 
        \begin{equation*}
            \|P_{N}^{x,y}(f\otimes g)\|_{L^\infty(\R^{n+m})} \lesssim \sum_{\max\{N_1, N_2\} \sim N} \|P_{N_1}f \|_{L^\infty(\R^n)}\|P_{N_2}g\|_{L^\infty(\R^m)}.
        \end{equation*}
        By estimating the $L^\infty$ norms under the sum by $N_1^{-s} N_2^{-t} \|f\|_{s}\|g\|_{t}$ we may continue the estimate by 
        \begin{equation*}
            \|P_{N}^{x,y}(f \otimes g)\|_{L^{\infty}(\R^{n+m})} \lesssim \left(N^{-(\max\{s, 0\} + t)} + N^{-(\max\{t, 0\} + s)}\right)\|f\|_{s}\|g\|_{t}.
        \end{equation*}
        But the term in parentheses is bounded by $N^{-\min\{s, t, s+t\}}$, which can be shown by a simple case analysis, which proves that if $f \in B^s(\R^n)$, $g \in B^t(\R^m)$, then the tensor product $f \otimes g \in B^r(\R^{n+m})$ for every $r < \min \{s, t, s+t\}$. 
        We now turn to the proof of the proposition. 
        \paragraph{\textbf{Analysis of $\cWF$}} For the first claim about the classical wavefront set, we take a point $(x_0, y_0; \xi_0, \eta_0)$ which does not belong to the right-hand side of \eqref{E:cwftensor}. We then subdivide into cases, as the only restriction is that $(\xi_0, \eta_0) \neq (0,0)$. We always assume that $x_0 \in \supp(f)$ and $y_0 \in \supp(g)$, otherwise the conclusion is trivial. 

        For the first case, we assume that $\xi_0\neq 0$ and $\eta_0 \neq 0$ and that neither $f$ nor $g$ is microlocally smooth at $(x_0, \xi_0), (y_0, \eta_0)$, respectively. Then nonmembership in the right-hand side of \eqref{E:cwftensor} means that both 
        \begin{equation*}
            (x_0, \xi_0) \notin \cWF_0^s(f) \qtq{and} (y_0, \eta_0) \notin \cWF_0^t(g)
        \end{equation*}
        simultaneously. But then we can find cutoffs $\vphi_{x_0}, \vphi_{y_0}$ to neighborhoods of $x_0, y_0$ and directional cutoffs $\psi_{\xi_0}$, $\psi_{\eta_0}$ so that 
        \begin{equation*}
            \vphi_{x_0}\psi_{\xi_0}(D_x) u \in B^s \qtq{and} \vphi_{y_0} \psi_{\eta_0}(D_y) v \in B^t.
        \end{equation*}
        We can now find a directional cutoff $\Psi(\xi, \eta)$ supported where both directional cutoffs $\psi_\xi, \psi_\eta$ are identically $1$; taking the tensorized cutoff $\vphi_{(x,y)} = \phi_x \otimes \vphi_y$ and using \Cref{L:sepsupps}, we find 
        \begin{equation*}
            \vphi_{(x,y)} \Psi(D_x, D_y) (f\otimes g) = \Psi(D_x,D_y) [\vphi_x \psi_{\xi}f \otimes \vphi_y \psi_\eta g] + \text{Schwartz},
        \end{equation*}
        which is in $B^r$ for $r < \min\{s, t, s+t\}$ by the Littlewood-Paley estimate above.

        By symmetry it next suffices to consider the case $\xi_0 = 0, \eta_0 \neq 0$. For this to happen, given $x_0 \in \supp(f)$, we must have $(y_0, \eta_0) \notin \WF(g)$. But now this means that we can find cutoffs $\vphi_{y_0}, \psi_{\eta_0}(D_y)$ so that $\vphi_{y_0}\psi_{\eta_0}(D_y)g \in \Sw(\R^m)$; tensorizing and using a similar cutoff $\Psi(D_x, D_y)$ as above yields the claim. 
        \paragraph{\textbf{Analysis of $\FWF$}}
        If we take a point $(x_0, y_0, \xi_0, \eta_0)$ which does not belong to the right-hand side of \eqref{E:FWFtensor} then we can find appropriate cutoffs $\vphi_{\xi_0}(D_x), \vphi_{\eta_0}(D_y)$ and directional cutoffs $\psi_{x_0}, \psi_{y_0}$ so that 
        \begin{equation*}
            \psi_{x_0}\vphi_{\xi_0}(D_x)f \in C_b^\infty(\R^n) \qtq{and} \psi_{y_0}\vphi_{\eta_0}(D_y)g \in C_b^\infty(\R^m).
        \end{equation*}
        But then tensorizing the frequency cutoffs and defining $\Psi$ to be a directional cutoff to a conical neighborhood where $\psi_{x_0}, \psi_{y_0}$ are both identically $1$, we obtain the result in the case that $x_0, y_0 \neq 0$. If either (but not both) is zero, we simply replace its directional cutoff with $1$ and repeat the above analysis. 
        \paragraph{\textbf{Analysis of $\oscWF$}} We begin again by taking a point $(x_0, y_0, \xi_0, \eta_0)$ outside the right-hand side of \eqref{E:oscwftensor}. A virtually identical argument to the cases above supplies four directional cutoffs $\psi_{x_0}, \psi_{\xi_0}(D_x)$, $\psi_{y_0}, \psi_{\eta_0}(D_y)$ so that both 
        \begin{equation*}
            \psi_{x_0}\psi_{\xi_0}(D_x)f \in B^s(\R^n)\qtq{and} \psi_{y_0}\psi_{\eta_0}(D_y)g \in B^t(\R^m)
        \end{equation*}
        both hold. We remark that in the edge cases where $x_0 = 0$, we take $\psi_{x_0} \equiv 1$; if $\xi_0 = 0$, we similarly take $\psi_{\xi_0} \equiv 1$, with similar conventions for $y_0, \eta_0$. We also note that in the case where $(x_0, \xi_0) = (0,0)$ we automatically assume that $f \in B^s(\R^n)$. But now rerunning the same arguments, choosing encompassing directional cutoffs $\Psi_{(x_0, y_0)}(x,y)$ and $\Psi_{(\xi_0,\eta_0)}(D_x, D_y)$, we arrive at the expression 
        \begin{equation*}
            \Psi_{(x_0, y_0)}\Psi_{(\xi_0, \eta_0)}(D_x, D_y) (f\otimes g) = \Psi_{(x_0, y_0)} \Psi_{(\xi_0, \eta_0)}(D_x, D_y) \left[\psi_{x_0}\psi_{\xi_0}(D_x)f \otimes \psi_{y_0}\psi_{\eta_0}(D_y)g\right].
        \end{equation*}
        The expression in brackets is an element of $B^r(\R^m \times \R^n), r < \min\{s, t, s+t\}$ by caveat, and the operators on the outside are bounded on every Besov space, completing the proof. 
    \end{proof}
    \section{Paraproducts and H\"ormander's Condition}\label{S:paraproducts}
        In this section, we use the tools from \Cref{S:Properties} to prove our main theorem. We begin with a short technical lemma, which involves intersecting the wavefront sets with the normal to the diagonal $N_\Delta$ and the diagonal at infinity $N_\Delta^\infty$. 
        \begin{lemma}\label[lemma]{L:refined_tensor_est}
            Let $r, s, t \in \R$. Then we have the following refinement of the tensor product containments above: 
            \begin{align}
                \cWF^r(f \otimes g) \cap \cNDelta &\subset \left[ \cWF^s(f) \boxtimes \WF(g) \cup \WF(f) \boxtimes \cWF^t(g)\right] \cap N_\Delta \label{E:refcwf}\\
                \oscWF^r(f \otimes g) \cap \oscNDelta & \subset[ \oscWF^s(f) \boxtimes \R^{2d} \cup \R^{2d} \boxtimes \oscWF^t(g)] \cap \oscNDelta\label{E:refoscwf}
            \end{align}
            for any $r < s+t$.
        \end{lemma}
        \begin{proof}
            We first note that none of the `zero' components of the enlarged wavefront sets in \Cref{P:tensorprod} may enter, owing to the intersections with $\cNDelta$ and $\oscNDelta$. We now make the same argument as in the proof of \Cref{P:tensorprod}, but additionally restrict the support of the outer directional cutoff $\Psi$ to the conical set $\{|\xi| \approx \eta\}$. Then at the stage of the Besov bounds, we have 
            \begin{equation*}
                \begin{split}
                    \| \bbar{\chi}_{(x_0, y_0)} \Psi_{(\xi_0, \eta_0)}(D_x, D_y) [\chi_{x_0} \psi_{\xi_0}(D_x) f\otimes \chi_{y_0}\psi_{\eta_0}(D_y) g] \|_{B^r} \\\lesssim \sup_N N^r \sum_{N_1 \approx N_2 \approx N} \|P_{N_1} \chi_{x_0} \psi_{\xi_0}(D_x) f\|_{L^\infty} \|P_{N_2} \chi_{y_0}\psi_{\eta_0}(D_y) g\|_{L^\infty}
                \end{split}
            \end{equation*}
            where $\chi \in \{\vphi, \psi\}$ and $\bbar{\chi} \in \{\Phi, \Psi\}$ are cutoffs of suitable type. But by hypothesis the terms in the $L^\infty$ norms are in $B^s$, $B^t$, respectively, which gives a bound on the above display by
            \begin{equation*}
                \begin{split}
                     &\sup_N N^r \sum_{N_1 \approx N_2 \approx N} N_1^{-s} N_2^{-t} \|\chi_{x_0} \psi_{\xi_0}(D_x) f\|_{B^s} \|\chi_{y_0}\psi_{\eta_0}(D_y) g\|_{B^t} \\
                     &\lesssim \sup_N N^{r - (s+t)} \|\chi_{x_0} \psi_{\xi_0}(D_x) f\|_{B^s} \|\chi_{y_0}\psi_{\eta_0}(D_y) g\|_{B^t},
                \end{split}
            \end{equation*}
            which is finite as long as $r - (s+t) < 0$, i.e. $r < s+t$.
        \end{proof}
        We now turn to the proof of the main theorem. 
        \begin{proof}[Proof of \Cref{T:maintheorem}]
            We begin by showing item $(3)$.  Take $f, g \in \Sw'(R^d)$; to pull back $\sP^{(2)}(f \otimes g)$ by $\Delta$, we need to show that $\cWF^r( \sP^{(2)}(f\otimes g) \cap \cNDelta = \emptyset$ and similarly that $\oscWF^r(\sP^{(2)}(f\otimes g)) \cap \oscNDelta = \emptyset$. However, this immediately follows from the fact that the symbol of $\sP^{(2)}$ is supported where $\{|\xi| \ll |\eta|\}$, which cannot meet either $\cNDelta$ or $\oscNDelta$; appealing to \Cref{R:settheory} gives the claim.
            
            We turn to the proof of items (1) and (2). We notice that once item (1) is proven, in view of 
            \begin{equation*} \label{E:paraproductexpn}
                fg = \Delta^{\ast}\left(\sP^{(2)}[f\otimes g] + \Pi^{(2)}[f \otimes g] + \sP^{(2)}[g \otimes f]\right),
            \end{equation*}
            we deduce item (2) immediately. To prove (1), we use the same argument as for $\sP$; however, the support of $\Pi^{(2)}  = \{|\xi| \approx |\eta|\}$ may meet $\cNDelta$ or $\oscNDelta$. However, we see that by \Cref{P:micro_paraproducts} and \Cref{L:microlocality_and_microellipticity} we need only to check that 
            \begin{equation*}
                \cWF^r(f \otimes g) \cap \cNDelta = \emptyset \qtq{and} \oscWF^r(f \otimes g) \cap \oscNDelta = \emptyset \qtq{for} r< \alpha+\beta.
            \end{equation*}
            This follows immediately from combining \cref{E:H1,E:H2}, the definitions of $\cNDelta$ and $\oscNDelta$ and \Cref{L:refined_tensor_est}.
            To see the sharpness, we take a Schwartz function $\vphi$ so that $P_1 (\vphi^2) = \vphi^2$, and consider frequencies $N_j = L^j$ for some sufficiently large $L$ so that $N_j \approx N_{j'}$ if and only if $j = j'$. Let $\hat{\xi} \in S^{d-1}$, and consider for $s, t \in \R$
            \begin{equation*}
                f_J(x) = \vphi(x) \sum_{j=1}^J N_j^{-s}e^{i N_j \hat{\xi} \cdot x} \qtq{and} g_J(x) = \vphi(x) \sum_{j= 1}^J N_j^{-t} e^{i N_j \hat{\xi} \cdot x}. 
            \end{equation*}
            One may check that $f_J, g_J$ have limits $f, g$ in $B^s, B^t$, respectively, and $\Pi^{(2)}[f_J \otimes g_J]$ has a limit in $\Sw'(\R^{2d})$. However, after pulling back by $\Delta$, we find 
            \begin{equation*}
                \Pi(f_J, g_J) = \vphi^2 \sum_{j=1}^J N_j^{-(s+t)} \to \infty \qtq{for} s + t \leq 0. 
            \end{equation*}
            Hence no continuous extension of the resonant product may exist for $s+t \leq 0$.
        \end{proof}
    \appendix
    \setcounter{definition}{0}
    \renewcommand{\thelemma}{\Alph{section}.\arabic{lemma}}
    
    \section{Auxiliary Facts}\label[appendix]{S:Appendix}
        In this Appendix, we record some auxiliary technical lemmas. 
        \begin{lemma}\label[lemma]{L:sepsupps}
            Let $a_1, a_2 \in C^\infty(\R^d \to [0, 1])$ with bounded derivatives of all orders and suppose that $b$ is either a frequency cutoff or a directional frequency cutoff. Suppose that $\supp(a_1), \supp(a_2)$ are disjoint in the sense that 
            \begin{equation}
                |z_1 - z_2| \gtrsim (1+ |z_1| + |z_2|) \qtq{for all} z_1 \in \supp(a_1),  z_2 \in \supp(1-a_2). 
            \end{equation}
            Then the operator $a_1 b(D) (1-a_2) \: \Sw'(\R^d) \to \Sw(\R^d)$ is a smoothing operator. 
        \end{lemma}
        \begin{proof}
            Let $K$ be the Schwartz kernel of the operator $b$, so that 
            \begin{equation*}
                b(D) f(x) = \int_{\R^d} K(x-y)f(y) \dy \qtq{for all} f \in\Sw(\R^d).
            \end{equation*}
            Then writing $K$ as a Fourier transform and integrating by parts using the hypotheses on $b$, we see that 
            \begin{equation*}
                |\partial^{\alpha} K(x)| \lesssim \jbrak{x}^{-M} \qtq{for all} M \geq 0. 
            \end{equation*}
            From here it is straightforward to see using the conditions on the support of $a_1, a_2$ that $a_1(x) K(x-y)(1-a_2(y))$ is a Schwartz function on $\R^{2d}$, which implies the operator $a_1 b(D)(1-a_2)$ is smoothing.
        \end{proof}

        \begin{lemma}[Further Localization]
            \label[lemma]{L:further-localization}
            Let $B$ be an SG-compatible Banach space, $u \in \D'$, and $a \in \scS^{0,0}$. Then for any $b \in \scS^{0,0}$ with support contained in $\supp a$ such that $d(\supp b, (\supp a)^c) > \gamma > 0$, $\Op(a)u \in B \implies \Op(b)u \in B$.
        \end{lemma}
        \begin{proof}
            On $\supp b$, $a$ is elliptic. Thus, we can construct a parametrix $c \in \scS^{0,0}$ such that the composition $\Op(c) \Op(a) = I$ up to an infinitely smoothing remainder. But then, $\Op(b)u = \Op(b) \Op(c) \Op(a) u$ up to a Schwartz error. The hypothesis, combined with \Cref{L:microlocality_and_microellipticity}, tell us that $\scWF(\Op(b)\Op(c) \Op(a) u)$ is contained in $\scWF(\Op(a)u) = \varnothing$.
        \end{proof}
    \bibliography{references}
    \bibliographystyle{bjoern_style}
\end{document}